\documentclass[12pt]{amsart}
\usepackage{epsfig,color}

\usepackage{mathrsfs}

\usepackage{comment}

\usepackage{url}

\theoremstyle{definition}

\newtheorem{theorem}{Theorem}[section]
\newtheorem{definition}[theorem]{Definition}

\newtheorem{proposition}[theorem]{Proposition}
\newtheorem{lemma}[theorem]{Lemma}
\newtheorem{remark}[theorem]{Remark}
\newtheorem{corollary}[theorem]{Corollary}

\numberwithin{equation}{section}

\usepackage{enumerate}

\DeclareMathOperator*{\supp}{spt}

\newcommand{\lc}{\scalebox{1.8}{$\llcorner$}}
\newcommand{\btau}{\boldsymbol\tau}

\newcommand{\mc}{\mathcal}

\DeclareMathOperator*{\area}{Area}
\DeclareMathOperator*{\length}{Length}

\DeclareMathOperator*{\Div}{div}

\usepackage{comment}

\title[Min-max theory for CMC networks]{Min-max theory for networks of constant geodesic curvature}

\author[Xin Zhou]{Xin Zhou}
\address{Department of Mathematics, University of California Santa Barbara, Santa Barbara, CA 93106, USA; and School of Mathematics, Institute for Advanced Study, Princeton, NJ 08540, USA}
\email{zhou@math.ucsb.edu}

\author[Jonathan J. Zhu]{Jonathan J. Zhu}
\address{Department of Mathematics, Princeton University, Princeton, NJ 08544, USA}
\email{jjzhu@math.princeton.edu}

\begin{document}
\begin{abstract}
We prove that on a closed surface, for any $c>0$, our min-max theory for prescribing mean curvature produces a solution given by a curve of constant geodesic curvature $c$ which is almost embedded, except for finitely many points, at which the solution is a stationary junction with integer density. Moreover, each smooth segment has multiplicity one. The key is a classification of blowups which is new even for $c=0$. 
\end{abstract}
\maketitle

\section{Introduction}

The min-max construction of closed geodesics dates back to Birkhoff \cite{Birkhoff17}, and tremendous progress has been made since then (see \cite{MN17} for a nice summary). The min-max construction of closed curves (or networks) of (nonzero) constant geodesic curvature, however, has not been thoroughly investigated. In particular, it has been conjectured by Arnold \cite[page 395]{Arnold04} and Novikov \cite[Section 5]{Arnold04} that every topological two sphere admits closed embedded curves of any prescribed constant geodesic curvature. This conjecture remains open, and we refer to \cite{Ginzburg96, Rosenberg-Schneider11, Schneider11} for more background and some partial results towards this conjecture.

The goal of this article is to show that on a closed surface, for any $c>0$, our CMC min-max theory \cite{Zhou-Zhu17, Zhou-Zhu18} (which is based on the Almgren-Pitts min-max theory for minimal hypersurfaces \cite{AF62, Pitts81}) produces a solution given by a curve of constant geodesic curvature $c$ which is almost embedded, except for finitely many points, 
at which the solution is a stationary junction with integer density. Moreover, each smooth constant geodesic curvature segment has multiplicity one. See Theorem \ref{T:main result1} for a precise statement.


The key is a graph theoretic argument (Section \ref{S:combinatorial argument}) to classify blowups which have 
a number of iterated replacements in open disks. In particular, we prove that such blowups are integer multiple of a line.  See Theorem \ref{T:main result2} for the detailed statement.


This classification result (for blowups) is new even for the case of geodesics, that is $c=0$. The existence of a nontrivial geodesic network was known by Pitts \cite{Pitts74} (based upon earlier works of Almgren \cite{AF65}; note that Pitts' result also holds true in higher codimension). In fact, in \cite{Pitts74} Pitts proved that the 1-dimensional min-max varifold is always supported in the image of its tangent varifold under the exponential map at any given point. Consequently, the min-max varifold is represented by a geodesic network. Note that Pitts's result does not preclude the tangent varifold being a bouquet of half lines (even if the min-max varifold is almost minimising near that given point). 

The existence of a geodesic network has another proof by combining Pitts \cite{Pitts81} with Allard-Almgren \cite{Allard-Almgren76}. Pitts \cite{Pitts81} proved the existence of a weak min-max solution as a nontrivial, stationary, integer rectifiable, 1-varifold in any closed manifold. The regularity theory of Allard-Almgren \cite{Allard-Almgren76} for stationary 1-varifolds then implies that Pitts's weak solution is a geodesic networks (with constant integer multiplicity on each geodesic segment). We also refer to Calabi-Cao \cite[Appendix]{Calabi-Cao92} and Aiex \cite{Aiex16} for other proofs of this result. 

However, even on surfaces, one can not follow Pitts's regularity argument  in \cite{Pitts81} (which succeeds for hypersurfaces of dimension between 2 and 6) to prove this network regularity without Allard-Almgren \cite{Allard-Almgren76}. The main missing ingredient for curves is that Simons' classification for minimal stable hypercones \cite{Simons68} does not hold for curves. In particular, Pitts's argument \cite[7.8]{Pitts81} cannot extend to prove that tangent cones are lines without Simons' classification. 

Nevertheless, using our new characterisation for blowups, Pitts's work \cite{Pitts81} does directly imply the geodesic network regularity of his weak solution. In fact, away from finitely many points, Pitts' weak solution has the good replacement property in small balls, so any tangent varifold satisfies the assumptions of our classification result (using an observation in \cite[Lemma 5.10]{Zhou-Zhu17}), and hence is an integer multiple of a line. With this, one can proceed the same as Pitts to obtain the desired regularity. 

In this paper, we carry out the process described above in the setting for $c>0$, using the theory developed by us in \cite{Zhou-Zhu17, Zhou-Zhu18}.

\vspace{1em}
In Section \ref{S:Min-max construction for weighted length functional}, we introduce the problem and state the main result. In Section \ref{S:previous results and proof}, we collect necessary results in our previous CMC min-max theory and prove the main theorem. In Section \ref{S:combinatorial argument}, we prove the key ingredient on classifying blowups.

\begin{remark}
In \cite{KL18}, the authors have built upon our results, improving the regularity to show that the networks produced are either smooth or $C^{1, 1}$ curves. In particular, they proved that if the only junction is not a smooth point, then the tangent cone consists of two lines intersecting transversally.  



\end{remark}

\section*{Acknowledgements}

X. Zhou is partially supported by NSF grant DMS-1811293.  J. Zhu is supported by NSF grant DMS-1802984, and would like to thank Erick Knight for helpful discussions. Both authors would like to thank the Institute for Advanced Study for their hospitality.

\section{Min-max construction for weighted length functional}
\label{S:Min-max construction for weighted length functional}

In this part, we will briefly introduce the setup for min-max construction for constant geodesic networks. We refer to \cite{Zhou-Zhu17} for more details. 

Let $(S, g)$ be a closed 2-dimensional surface with a Riemannian metric $g$. Fix a positive number $c>0$. Given any Caccioppoli set $\Omega\subset S$, we define the {\em $c$-weighted length functional} or {\em $c$-length} as
\begin{equation}
\label{D:Lc functional}
{\mc L}^c(\Omega)=\length(\partial\Omega)-c\area(\Omega),
\end{equation}
where $\length$ and $\area$ are calculated with respect to the metric $g$.

A 1-parameter families of Caccioppoli sets $\{\Omega_t\}_{t\in[0, 1]}$ is said to be a {\em sweepout}, if
\begin{itemize}
\item $\Omega_0=\emptyset$, $\Omega_1=\Sigma$;
\item the boundaries $\{\partial\Omega_t\}$ are continuous in $t$ with respect to the flat topology.
\end{itemize}  

We can then define the min-max value of  ${\mc L}^c$ as
\begin{equation}
\label{E:c-min-max}
{\bf L}^c=\inf\left\{\max_{t\in[0, 1]} {\mc L}^c(\partial\Omega_t):\, \{\Omega_t\}_{t\in[0, 1]} \text{ is a sweepout}\right\}.
\end{equation}

In this paper we will prove that
\begin{theorem}
\label{T:main result1}
There exists a nontrivial 1-varifold $V$, finitely many points $\{p_i\}_{i=1}^n\subset S$, and a Caccioppoli set $\Omega$, such that 
\begin{enumerate}
\item $V$ is induced by $\partial \Omega$ (of multiplicity 1);
\item away from $\{p_i\}_{i=1}^n$, the boundary $\gamma_0=\partial\Omega$ is an almost embedded curve of constant geodesic curvature $c$; 
\item at each $p_i$, the density of $V$ is an integer, and any tangent cone is a stationary geodesic network in $\mathbb R^2$, smooth away from $0$.
\end{enumerate}
\end{theorem}

Here `almost embedded' means that $\gamma_0$ is a smooth immersion, and near each self-intersection point $\gamma_0$ decomposes to two connected embedded components which touch but do not cross.

\begin{remark}
In fact, by refining Pitts's combinatorial argument \cite[4.10]{Pitts81} with the observation of Colding-De Lellis (the remark after \cite[Proposition 3.3]{Colding-DeLellis03}), one can show that the set $\{p_i\}_{i=1}^n$ consists of only one point. We will fill in the details of this fact elsewhere.
\end{remark}

\section{Results from \cite{Zhou-Zhu17} and proof of Theorem \ref{T:main result1}}
\label{S:previous results and proof}

In \cite{Zhou-Zhu17, Zhou-Zhu18}, the authors established an existence theory, which in this setting yields that there is a 1-varifold $V$ associated with $\bf L^c$ satisfying a list of useful properties that we will summarize in the following. In particular, the theory in \cite{Zhou-Zhu17, Zhou-Zhu18} works in any closed Riemannian manifold $(M^{n+1}, g)$ (using the corresponding $n$-dimensional $c$-weighted area functional), and when $3\leq n+1\leq 7$, we proved that $V$ is induced by the boundary of some Caccioppoli set $\Omega_0$, whose boundary $\Sigma_0=\partial\Omega$ is an almost embedded 
closed hypersurface of constant mean curvature $c$. However, since the classification of stable minimal hypercones by Simons \cite{Simons68} does not hold in dimension $n=1$, we cannot directly obtain similar regularity results for $V$ when $n=1$. Instead, we will exploit some stronger properties of $V$ that were obtained in \cite{Zhou-Zhu17} to achieve some partial regularity. In fact, we will use certain good replacement properties in small disks instead of just in small annuli.

Note that we used a discrete setup in \cite{Zhou-Zhu17, Zhou-Zhu18} following the classical work of Almgren-Pitts \cite{AF62, Pitts81}. We will not dip into these sophisticated notations, as we can start directly with the outcomes in \cite{Zhou-Zhu17}.

Before summarizing what we proved in \cite{Zhou-Zhu17}, we need to introduce the notion of $c$-replacements. A 1-varifold $V$ is said to have {\em $c$-bounded first variation}, if for any smooth vector field $X$ on $S$, 
\[ \left| \int {\Div}_S X(x) dV(x, S) \right|\leq c \cdot \int_S |X(x)|\, d\|V\|(x). \]

\begin{definition}
\label{D:c-replacement} 
Given a 1-varifold $V$ with $c$-bounded first variation and an open set $U$ in $S$, $V^*$ is said to be a {\em $c$-replacement} of $V$ in $U$ if 
\begin{enumerate}
\item $V$ coincides with $V^*$ outside the closure $\overline{U}$, i.e., $V\lc Gr_1(S\backslash \overline{U})=V^*\lc Gr_1(S\backslash \overline{U})$;\footnote{Here $Gr_1(U)$ is the Grassmannian bundle of 1-lines over $U$.}
\item $\|V\|(S)-c\cdot \area(U) \leq \|V^*\|(S) \leq \|V\|(S)+c\cdot\area(U)$;
\item $V^*$, when restricted to $U$, is induced by the boundary of some open subset $\Omega^*\cap U$ (here $\Omega^*$ is an Caccioppoli set), that is, $V^*\lc Gr_1(U)=[\partial \Omega^*\cap U]$, such that $\partial\Omega^* \cap U$ is an almost embedded curve of constant geodesic curvature $c$;
\item $V^*$ has $c$-bounded first variation.
\end{enumerate}
\end{definition}

We proved in \cite{Zhou-Zhu17} that $V$ has certain good replacement properties:

\begin{theorem}
\label{T:results from Zhou-Zhu17}
[Theorem 5.6, Proposition 5.8, Lemma 5.9 in \cite{Zhou-Zhu17}]
Given $c>0$, let ${\bf L}^c$ be defined as (\ref{E:c-min-max}), then there exists a 1-varifold $V$ in $(S, g)$, such that
\begin{enumerate}
\item ${\bf L}^c>0$ and hence $V$ is nontrivial;
\item $V$ has $c$-bounded first variation;
\item for any $p\in S$, $V$ has a $c$-replacement $V^*$ in any small enough annulus centered at $p$; hence by a covering argument, there exists a finite set $\mathcal{P}=\{p_i\}_{i=1}^n$, so that for any $p\in S\backslash \mathcal{P}$, there exists a neighborhood $U\subset S\backslash \mathcal{P}$ of $p$, such that $V$ has a $c$-replacement $V^*$ in $U$; 
\item in any neighborhood $U$ where $V$ has a $c$-replacement, $V^*$ also has a $c$-replacement $V^{**}$ in $U$; and this procedure of taking $c$-replacements can be iterated as many times as one wants.
\end{enumerate}
\end{theorem}
\begin{remark}
In \cite[Theorem 5.6, Proposition 5.8, Lemma 5.9]{Zhou-Zhu17}, we proved that $V$ is $c$-almost minimizing in any small annulus and hence has a $c$-replacement. As mentioned earlier, by a remark of Colding-De Lellis after \cite[Proposition 3.3]{Colding-DeLellis03}, one can prove that $V$ is $c$-almost minimizing in any small open neighborhood, except at one point. (This will be addressed elsewhere by the authors.) 

To gain regularity of $V^*$ in $U$, we used curvature estimates for stable hypersurfaces of constant mean curvature in \cite[Theorem 2.6]{Zhou-Zhu17}, but this is trivially true in dimension $n=1$ for curves of constant geodesic curvature.
\end{remark}

As a key step to obtain our main regularity results in \cite{Zhou-Zhu17}, we analyzed the blowups of $V$ using the good replacement properties. In particular, we proved,
\begin{proposition}
\label{P:blowups have good replacements}
[Lemma 5.10 in \cite{Zhou-Zhu17}] 
Let $V$ be as in Theorem \ref{T:results from Zhou-Zhu17}. Given any $p\in S\backslash \{p_i\}_{i=1}^n$, and a tangent varifold $C\in \text{TanVar}(V, p)$ of $V$ at $p$, then $C$ satisfies,
\begin{enumerate}
\item $C$ is a stationary $1$-varifold in $\mathbb{R}^2$;
\item given any open set $U\subset \mathbb{R}^2$, $C$ has a $0$-replacement $C^*$;
\item $C^*$ has $0$-replacement in any open set $W\subset \mathbb{R}^2$.
\end{enumerate}
\end{proposition}

As a direct corollary of Theorem \ref{T:main result2} in Section \ref{S:combinatorial argument}, we have,
\begin{corollary}
\label{C:tangent cones are lines}
Any $C$ in Proposition \ref{P:blowups have good replacements}  is an integer multiple of a line passing the origin.
\end{corollary}

\vspace{1em}
Now we are ready to sketch the proof of Theorem \ref{T:main result1}. Using Corollary \ref{C:tangent cones are lines} in place of \cite[Proposition 5.11]{Zhou-Zhu17}, the regularity of $V$ away from $\{p_i\}_{i=1}^n$ follows from that of \cite[Theorem 6.1]{Zhou-Zhu17} with minor modifications. The structure of tangent cones of $V$ at $\{p_i\}_{i=1}^n$ follows from a classical argument of characterizing tangent cones of min-max varifold by Almgren-Pitts \cite[3.13]{Pitts81}. 
We will mainly focus on the differences with the proof of \cite[Theorem 6.1]{Zhou-Zhu17}.

\begin{proof}[Proof of Theorem \ref{T:main result1}] We will prove parts (1)(2)(3) in three steps.
\vspace{1em}

\textbf{Step 1:} We first focus on a neighborhood of a point $p\in \textrm{spt}\|V\| \backslash \mathcal{P}$, where the set $\mathcal{P}=\{p_i\}_{i=1}^n$ is given in Theorem \ref{T:results from Zhou-Zhu17}. Take a small enough radius $r>0$, such that $V$ has $c$-replacements in the geodesic ball $B_r(p)\subset S$. Fix any $0<s<r$, and take a $c$-replacement $V^*$ in the annulus $A_{s, r}(p)=B_r(p)\backslash\overline{B_s(p)}$. By the definition of $c$-replacement, $V^*\lc A_{s, r}(p)$ is induced by the boundary of some Caccioppoli set $\Omega^*$, and is an almost embedded curve, denoted by $\gamma_1$, of constant geodesic curvature $c$.

Take a radius $s<t<r$, such that the sphere $\partial B_{t}(p)$ intersects $\gamma_1$ transversally\footnote{The existence of such $t_2$ follows from Sard's Theorem.}, and intersects along the regular (non-touching) set of $\gamma_1$\footnote{The touching set of $\gamma_1$ is a discrete set.}. Now take a $c$-replacement $V^{**}$ of $V^*$ in $B_t(p)$ (usually called the second replacement). Again $V^{**}\lc B_t(p)$ is given by an almost embedded curve $\gamma_2$ of constant geodesic curvature $c$. Using Corollary \ref{C:tangent cones are lines} in place of \cite[Proposition 5.11]{Zhou-Zhu17}, we can follow the same procedure as in \cite[Theorem 6.1, Steps 1 and 2]{Zhou-Zhu17} to show that $\gamma_1=\gamma_2$ in the overlapping region $A_{s, t}(p)$, and hence they form an almost embedded curve $\gamma$ in $B_r(p)$. 

The next step is to use $c$-replacements in annuli $A_{\tau, t}(p)$, where $0<\tau<s$. Let $V^{**}_{\tau}$ be the $c$-replacement of $V^*$ in $A_{\tau, t}(p)$, which is induced by an almost embedded curve $\gamma_\tau$.  By the same reasoning, we have $\gamma_\tau=\gamma_1$ in $A_{s, t}(p)$, and hence by ODE uniqueness theory, $\gamma_\tau=\gamma\cap A_{\tau, t}(p)$. 

Then by the moving sphere argument \cite[Theorem 6.1, Step 5]{Zhou-Zhu17}, we can show that $V$ is induced by $\gamma$ inside $B_s(p)$. This finishes the proof of the regularity of $V$ away from $\{p_i\}_{i=1}^n$ (part (2) in Theorem \ref{T:main result1}).

\vspace{1em}
\textbf{Step 2:}
By the same argument as \cite[Proposition 7.3]{Zhou-Zhu18}, $V$ is induced by the boundary of some Caccioppoli set $\Omega$, and 
\[ \mc L^c(\Omega)=\bf L^c. \]
This finished part (1) in Theorem \ref{T:main result1}.

\vspace{1em}
\textbf{Step 3:}
Finally we prove the structure of tangent cones $\textrm{TanVar}(V, p_i)$ at each $p_i$, i.e. part (3) in Theorem \ref{T:main result1}. Given a tangent cone $C$ at $p_i$, we know that $C$ is stationary and integer rectifiable, since $V$ has $c$-bounded first variation and is integer rectifiable. 
Now by smooth convergence, since $V$ consists of constant curvature curves, $C$ must be a geodesic network with constant integer multiplicity in each segment. Since $C$ is a cone, $\textrm{spt}\|C\|$ must be a finite union of half lines coming out of the origin. The only thing left to prove is to show that the sum of all integer multiplicities must be an even number, hence the density of $C$ at the origin - which is the same as that of $V$ at $p_i$ - is an integer.

Write 
\[ C=\lim_{j\to\infty} (\btau_{r_j, p_i})_{\#}V, \text{ as varifolds. }\]
Here $\{r_j\}$ is a sequence of positive numbers converging to $0$, and $\btau_{r_j, p_i}(x)=\frac{x-p_i}{r_j}$ are the rescaling maps\footnote{Here we can isometrically embed $(S, g)$ into some Euclidean space $\mathbb R^L$, and the calculation $\frac{x-p_i}{r_j}$ is done in $\mathbb R^L$.}.

Note that $V=\partial\Omega$, and consider the limit 
\[ \Omega'= \lim_{j\to\infty} (\btau_{r_j, p_i})_{\#}\Omega, \text{ as Cacciopolli sets. }\]
By the weak convergence, the $\supp \|\partial \Omega'\| \subset \supp \|C\|$, and it is easy to see that away from the origin the multiplicity of $C$ minus the multiplicity of $\partial\Omega'$ (which is identical to 1) must be an even number. On the other hand, one can see that there must be even numbers of half lines in $\textrm{spt}\|\partial\Omega'\|$ (to form the boundary of a set). Summing all ingredients together, we have proven that the number of half lines of $C$ (counting multiplicity) is even. 
\end{proof}

\section{Combinatorial argument}
\label{S:combinatorial argument}

In this part, we change gear to study geodesic networks arising in Proposition \ref{P:blowups have good replacements}. our main goal is to prove Theorem \ref{T:main result2}. 

We define a stationary network $V$ in $\mathbb{R}^2$ to be a network whose edges $vw$ are straight line segments with positive integer weight (multiplicity) $m_{vw}$, and which satisfies at each vertex $v$ of $V$ the stationarity condition \[ \sum_{vw \in V} m_{vw} \vec{T}_{vw} =0.\] Here $\vec{T}_{vw} = \frac{w-v}{|w-v|}$ is the outward unit tangent from $v$ along the edge $vw$. 

By a slight abuse of notation we henceforth consider stationary networks $V$ with $N$ vertices lying on the unit circle $S^1\subset \mathbb{R}^2$, each with an exterior radial edge to infinity; and $E$ edges interior to the circle. In what follows let $\mathring{V}$ be the interior graph of $V$ (consisting of those edges inside the circle); any graph theoretic concepts (degree, neighbourhood, etc.) are with respect to $\mathring{V}$. 

We say that such a stationary network $V$ is admissible if it satisfies: 
\begin{enumerate}
\item At each vertex $v$, we have \[m_v v + \sum_{w \in N_v} m_{vw} \vec{T}_{vw} =0,\] where $N_v$ is the set of vertices adjacent to $v$. Note as before $\vec{T}_{vw} = \frac{w-v}{|w-v|}$, and of course the radial edge has unit tangent $\vec{T}_v = v$. (This is just a restatement of stationarity, clarifying the notation for the exterior edges.) 
\item There are no crossings between interior edges. 
\end{enumerate}

We say that $V$ is a replacable network if it additionally satisfies the replacement property: 
\begin{enumerate}
\item[(3)] At each vertex $v$ in $V$, there is a replacement $V'_v$; that is, an admissible network $V'_v$ with exterior edges given by $\vec{T}_{vw}$ and multiplicity $m_{vw}$, for each $w\in N_v$. 
\end{enumerate}

Finally, we say that $V$ is a \textit{good network} if it is a replaceable network, each replacement $V'$ of $V$ is also replaceable, and so forth, so that $V$ has arbitrarily many iterated replacements. In fact we will only use four iterated replacements - two to rule out $N=3$, another to rule out $N=4$ and the fourth to rule out $N\geq 5$. 

\begin{lemma}
Let $f(N)$ be the maximum number of straight line segments that can be drawn between $N$ distinct points on the unit circle, which do not connect adjacent vertices, and do not have any crossings. Then $f(N)=\max(N-3,0)$. 

Consequently, the total number of interior edges in an admissible network is bounded by \[E\leq F(N) := \begin{cases} 2N-3&, N\geq 3\\ \max(N-1,0)&, N\leq 2 .\end{cases}.\] 
\end{lemma}
\begin{proof}
Any such edge divides the remaining vertices into a set of $k$ vertices and a set of $l$ vertices, $k,l\geq 1$. Then we have the recursive formula \[f(N) = \max\{1+f(k+2) +f(l+2) | k+l=N-2, k,l\geq 1\}.\] It is clear that $f(1)=f(2)=f(3)=0$. A straightforward induction then shows that $f(N)= N-3$ for all $N\geq 3$. 
\end{proof}

Note that if $V$ is an admissible network and any vertex $v$ has degree 1 in $\mathring{V}$, then the interior edge must be the diameter through $v$. Since no interior edges may cross, this implies that at most two vertices can have degree 1 ($v$ and its antipode). Indeed, we have

\begin{lemma}
Let $V$ be an admissible network and a vertex $v$ of (interior) degree 1. Then the number of interior edges is bounded by \[E\leq F_1(N) = \begin{cases} 2N-5 &, N\geq 4 \\ 1 &, N\leq 3\end{cases}.\]
\label{lem:deg1}
\end{lemma}
\begin{proof}
As above, the interior edge from $v$ must be a diameter of the circle and its antipode $w$ must be a vertex in $V$. The diameter $vw$ splits the remaining vertices into two sets of $k$ and $l$ vertices, where $k+l=N-2$ and without loss of generality $0\leq k\leq l$. If $w$ also has degree 1, then $E\leq 1+F(k)+F(l)$. 

Otherwise, $w$ has a second incident edge with positive weight, so to satisfy stationarity it must be connected by a third edge to the other side of $vw$. In particular we must have $k,l\geq 1$, and $E\leq 1+F(k+1)+F(l+1)$. 

Thus we have three cases: $k=0$, in which case $w$ must have degree 1 and \[E\leq 1 + F(N-2);\] $k=1$, in which case $N\geq 4$ and \[E \leq  1+ F(2)+F(N-2)=2 + F(N-2);\] finally $2 \leq k\leq l$ in which case \[E\leq 1+F(k+1)+F(l+1)= 2N-5.\] The result follows by the cases for $F(N-2)$. 
\end{proof}

\subsection{The case $N=3$.}

Let $V$ be an admissible network with vertices $v_j = e^{i\theta_j}$. 


Set $\alpha_{jk} = \theta_k-\theta_i$. Note that \[\vec{T}_{jk} = \frac{e^{i\theta_k} - e^{i\theta_j}}{|e^{i\theta_k} - e^{i\theta_j}|} = e^{i\theta_j} \frac{e^{i\alpha_{jk}}-1}{|e^{i\alpha_{jk}}-1|}.\] Also note that for $\theta \in [0,2\pi]$, we have \[e^{i\theta}-1 =2\sin \frac{\theta}{2} ie^{i\theta/2} =2\sin \frac{\theta}{2} e^{i\frac{\theta+\pi}{2}} ,\]  \[e^{-i\theta}-1 =-2\sin \frac{\theta}{2} ie^{-i\theta/2}.\]  

Then stationarity at each vertex $v_j$ gives (after dividing through by $e^{i\theta_j}$ respectively)
\begin{equation}\label{eq:S1} m_1 + m_{12}ie^{i\alpha_{12}/2}+m_{13}ie^{i\alpha_{13}/2}=0, \end{equation}
\begin{equation}\label{eq:S2} m_2 -m_{12}ie^{-i\alpha_{12}/2}+m_{23}ie^{i\alpha_{23}/2}=0, \end{equation}
\begin{equation}\label{eq:S3} m_3 -m_{13}ie^{-i\alpha_{13}/2}-m_{23}ie^{-i\alpha_{23}/2}=0. \end{equation}

Note that all vertices must have degree 2 (that is, $m_{jk}>0$). (Otherwise, exactly one vertex has degree 1, but then $2e = \sum \deg(v) = 5$ which is impossible.) 

For each vertex $v_j$, it is geometrically clear that the other two vertices cannot lie on the same side of the diameter through $v_j$, or else it would be impossible to satisfy the stationarity. Therefore $\alpha_{12} \in (0,\pi)$, $\alpha_{23}\in(0,\pi)$, $\alpha_{13} \in (\pi,2\pi)$. (In particular $\alpha_{12}\neq \pi$, since then the only way to satisfy stationarity at $v_1$ would be $m_{13}=0$, which cannot happen; and similarly $\alpha_{23}, \alpha_{13}\neq \pi$.)

Set $s_{jk} = \sin \frac{\alpha_{jk}}{2}$ and $c_{jk}=\cos\frac{\alpha_{jk}}{2}$. Note $s_{jk}, c_{12},c_{23}, -c_{13} \in (0,1)$. 

We may rewrite the real part of the system above as 
\begin{equation}
\label{eq:real}
\begin{pmatrix} m_1 \\ m_2 \\ m_3 \end{pmatrix} + \begin{pmatrix} -s_{12} & -s_{13} & 0 \\ -s_{12} & 0 &   -s_{23} \\ 0 & -s_{13} & -s_{23} \end{pmatrix} \begin{pmatrix} m_{12} \\ m_{13} \\ m_{23} \end{pmatrix} =0 
\end{equation}
and the imaginary part as
\begin{equation}
\begin{pmatrix} c_{12} & c_{13} & 0 \\ -c_{12} & 0 & c_{23} \\ 0 & -c_{13} & -c_{23} \end{pmatrix} \begin{pmatrix} m_{12} \\ m_{13} \\ m_{23} \end{pmatrix} =0.
\end{equation}

Since the $c_{ij}$ are nonzero, the matrix $C=\begin{pmatrix} c_{12} & c_{13} & 0 \\ -c_{12} & 0 & c_{23} \\ 0 & -c_{13} & -c_{23} \end{pmatrix}$ has rank 2, nullity 1 and one can verify that the kernel is spanned by $\begin{pmatrix} c_{13}c_{23} \\ -c_{12}c_{23} \\ c_{12}c_{13}\end{pmatrix}$.

\begin{lemma}
\label{lem:rational}
Let $V$ be an admissible network with $N=3$. Then $e^{i\alpha_{jk}}$ are rational points on the unit circle. 
\end{lemma}
\begin{proof}
By the characterisation of the kernel, we have $\begin{pmatrix} m_{12} \\ m_{13}\\ m_{23}\end{pmatrix}= \beta \begin{pmatrix} c_{13}c_{23} \\ -c_{12}c_{23} \\ c_{12}c_{13}\end{pmatrix}$ for some $\beta\neq 0$. The plugging into (\ref{eq:real}) we have 

\begin{equation}
\begin{split}
-\frac{1}{\beta}\begin{pmatrix}m_1 \\ m_2 \\ m_3 \end{pmatrix} &= \begin{pmatrix} -s_{12} & -s_{13} & 0 \\ -s_{12} & 0 &   -s_{23} \\ 0 & -s_{13} & -s_{23} \end{pmatrix} \begin{pmatrix} c_{13}c_{23} \\ -c_{12}c_{23} \\ c_{12}c_{13}\end{pmatrix} 
\\&= \begin{pmatrix} c_{23}s_{23} \\ -c_{13}s_{13} \\ c_{12}s_{12}\end{pmatrix},
\end{split}
\end{equation}
where in the last line we have used the trigonometric addition formulae. Considering the quotients $\frac{m_1 m_{23}}{m_{12}m_{13}}$ and so forth, it follows that each $\tan \frac{\alpha_{jk}}{2}$ is rational and hence $e^{i\alpha_{jk}}$ is a rational point. 
\end{proof}

\begin{proposition}
There is no good network $V$ with $N=3$.
\end{proposition}
\begin{proof}
Suppose $V$ is a good network with $N=3$. First take a replacement $V'_1$ of $V$ at $v_1$. Using that $e^{i\theta}-1 = 2\sin\frac{\theta}{2} e^{i\frac{\theta+\pi}{2}}$ for $\theta\in[0,2\pi]$, the replacement network $V'_1$ should have vertices $v'_1 = v_1$, $v'_2 = \vec{T}_{12}$, $v'_3 = \vec{T}_{13}$ (up to a coordinate rotation); the corresponding angle differences are $\alpha^{(1)}_{12} = \frac{\alpha_{12}+\pi}{2}$, $\alpha^{(1)}_{13} = \frac{\alpha_{13}+\pi}{2}$, $\alpha^{(1)}_{23} = \frac{\alpha_{23}}{2}$. Now consider the iterated replacement $V''_{11}$ at $v'_1$; then the angle differences will become $\alpha^{(11)}_{12} = \frac{\alpha_{12}+3\pi}{4}$, $\alpha^{(11)}_{13} = \frac{\alpha_{13}+3\pi}{4}$, $\alpha^{(11)}_{23} = \frac{\alpha_{23}}{4}$. 

Apply the same twice iterated replacement at $v_3$; then in particular the (non-reflex) opposite angle difference halves twice, so the resulting network $V''_{33}$ will have $\alpha^{(33)}_{12} = \frac{\alpha_{12}}{4}$. By Lemma \ref{lem:rational} we must have $e^{i\alpha}\in\mathbb{Q}(i)$ for each of these $\alpha$. Since $\mathbb{Q}(i)$ is a field this implies that $\exp(i(\alpha^{(11)}_{12} - \alpha^{(33)}_{12})) = e^{3i\pi/4} = \frac{1}{\sqrt{2}}(-1+i)$ is a rational point, which is absurd. 
\end{proof}

\subsection{General $N$}

Recall that $E,\deg,\cdots$ denote the edges, degree, etc. of the interior graph $\mathring{V}$.

\begin{proposition}
There is no good network $V$ with $N=4$. 
\end{proposition}
\begin{proof}
Suppose $V$ is a good network with $N=4$. By the previous proposition, there are no good networks with $N=3$, so by taking a replacement we see that each vertex $v$ in $\mathring{V}$ cannot have degree equal to 2. 

If there is a vertex with degree 1, then by Lemma \ref{lem:deg1}, we have that $E\leq F_1(4)=3$. But since there are at most two vertices with degree 1, and the remaining vertices must have degree at least 3, we have \[E=\frac{1}{2}\sum\deg(v)  \geq \frac{1}{2}(3(N-2)+2) = 4,\] which is a contradiction. 

On the other hand, if there are no vertices with degree 1, then they all have degree at least 3, so $\sum \deg(v) \geq 3N = 12$, but this contradicts the earlier bound \[\sum \deg v = 2E \leq 2F(4) = 10.\] 
\end{proof}

\begin{theorem}
\label{T:main result2}
The only good network $V$ is the straight-line network which has $N=2$, diametrically opposite vertices and equal multiplicities. 
\end{theorem}
\begin{proof}
It is clear that there is no admissible network with $N=1$, and the only admissible network with $N=2$ is the straight line configuration. 

Since we have proven there are no good networks with $N=3,4$, by taking replacements we have ruled out any vertex $v$ in $\mathring{V}$ having degree 2 or 3, and we may assume $N\geq 5$. 

If there is a vertex with degree 1, then by Lemma \ref{lem:deg1}, we have that $E\leq F_1(N)=2N-5$. But since there are at most two vertices with degree 1, and the remaining vertices must have degree at least 4, we have \[E=\frac{1}{2}\sum\deg(v)  \geq \frac{1}{2}(4(N-2)+2) = 2N-3,\] which is a contradiction. 

On the other hand, if there are no vertices with degree 1, then all vertices must have degree at least 4 and so $\sum \deg(v) \geq 4N$. But \[\sum\deg(v) = 2E \leq 2F(N) = 4N-6,\] which is a contradiction. 
\end{proof}

\begin{remark}
The stationarity of $V$ at each vertex automatically implies that:
\begin{enumerate}
\item $V$ is stationary at infinity, that is, \[\sum_v m_v v = 0.\] 
\item The mass of the interior graph $\mathring{V}$ is the same as the graph which extends the exterior rays into the origin, that is, \[ \sum_v m_v = \sum_{vw \in \mathring{V}} m_{vw} |v-w|.\] 
\end{enumerate}


\end{remark}

\bibliography{reference}
\bibliographystyle{plain}

\end{document}